\documentclass[11pt, reqno]{amsart}
\usepackage{amssymb,amsmath,amsfonts,latexsym,mathtools}
\usepackage{bm, enumerate}
\usepackage{graphicx}
\usepackage{color,soul}
\usepackage{hyperref}
\usepackage{dsfont}

\newcommand{\newsection}[1]{\setcounter{equation}{0} \section{#1}}
\newcommand{\bea}{\begin{eqnarray}}
\newcommand{\eea}{\end{eqnarray}}

\newcommand{\clb}{\mathcal{B}}

\newcommand{\cle}{\mathcal{E}}
\newcommand{\clf}{\mathcal{F}}
\newcommand{\clg}{\mathcal{G}}
\newcommand{\clh}{\mathcal{H}}
\newcommand{\clk}{\mathcal{K}}

\newcommand{\clm}{\mathcal{M}}

\newcommand{\cls}{\mathcal{S}}

\newcommand{\D}{\mathbb{D}}

\newcommand{\C}{\mathbb{C}}

\newcommand{\z}{\bm{z}}
\newcommand{\w}{\bm{w}}
\newcommand{\sfw}{\mathsf{w}}

\def \qed {\hfill \vrule height6pt width 6pt depth 0pt}
\def\textmatrix#1&#2\\#3&#4\\{\bigl({#1 \atop #3}\ {#2 \atop #4}\bigr)}
\def\dispmatrix#1&#2\\#3&#4\\{\left({#1 \atop #3}\ {#2 \atop #4}\right)}
\newcommand{\be}{\begin{equation}}
\newcommand{\ee}{\end{equation}}
\newcommand{\ben}{\begin{eqnarray*}}
\newcommand{\een}{\end{eqnarray*}}

\newcommand{\bi}{\begin{itemize}}
\newcommand{\ei}{\end{itemize}}

\newcommand{\x}{\mathbf x}
\newcommand{\fu}{\mathfrak u}
\newcommand{\fv}{\mathfrak v}
\newcommand{\ft}{\mathfrak t}

\newtheorem{Theorem}{\sc Theorem}[section]
\newtheorem{Lemma}[Theorem]{\sc Lemma}
\newtheorem{Proposition}[Theorem]{\sc Proposition}
\newtheorem{Corollary}[Theorem]{\sc Corollary}
\newtheorem{Definition}[Theorem]{\sc Definition}
\newtheorem{Example}[Theorem]{\sc Example}

\newtheorem{Remark}[Theorem]{\sc Remark}

\newtheorem{Note}[Theorem]{\sc Note}
\newtheorem{Question}{\sc Question}
\newtheorem{ass}[Theorem]{\sc Assumption}
\newcommand{\bt}{\begin{Theorem}}
\def\beginlem{\begin{Lemma}}
\def\beginprop{\begin{Proposition}}
\def\begincor{\begin{Corollary}}
\def\begindef{\begin{Definition}}
\def\beginexamp{\begin{Example}}
\def\beginrem{\begin{Remark}}
\def\beginq{\begin{Question}}
\def\beginass{\begin{ass}}
\def\beginnote{\begin{Note}}
\newcommand{\et}{\end{Theorem}}
\def\endlem{\end{Lemma}}
\def\endprop{\end{Proposition}}
\def\endcor{\end{Corollary}}
\def\enddef{\end{Definition}}
\def\endexamp{\end{Example}}
\def\endrem{\end{Remark}}
\def\endq{\end{Question}}
\def\endass{\end{ass}}
\def\endnote{\end{Note}}

\begin{document}

\title[de Branges-Rovnyak spaces which are complete Nevanlinna-Pick spaces]{de Branges-Rovnyak spaces which are complete Nevanlinna-Pick spaces}

\author[H. Ahmed]{Hamidul Ahmed}
\address{Department of Mathematics, Indian Institute of Technology Bombay, Powai, Mumbai, 400076, India}
\email{hamidulahmedslk@gmail.com, 22D0785@iitb.ac.in}

\author[B.K. Das]{B. Krishna Das}
\address{Department of Mathematics, Indian Institute of Technology Bombay, Powai, Mumbai, 400076, India}
\email{bata436@gmail.com, dasb@math.iitb.ac.in}

\author[S. Panja]{Samir Panja}
\address{Statistics and Mathematics Unit, Indian Statistical Institute, 8th Mile, Mysore Road, Bangalore, 560059, India}
\email{samirpanja\textunderscore pd@isibang.ac.in, panjasamir2020@gmail.com}


\begin{abstract}
We consider de Branges-Rovnyak spaces of a considerably large class of reproducing kernel Hilbert spaces and find a characterization for them to be complete Nevanlinna-Pick spaces. This extends earlier characterizations obtained for the Hardy space over the unit disc (\cite{Chu}) as well as for the Drury-Arveson space over the unit ball (\cite{Jesse}). Our characterization takes a complete form for the particular cases of the Hardy space over the polydisc and the Bergman space over the disc. We show that a non-trivial de Branges-Rovnyak space, associated to a contractive multiplier, of the Hardy space over the bidisc or the Bergman space over the unit disc is a complete Nevanlinna-Pick space if and only if it is isometrically isomorphic to the Hardy space over the unit disc as reproducing kernel Hilbert spaces. On the contrary, it is shown that non-trivial de Branges-Rovnyak spaces of the Hardy space over the $n$-disc with $n\ge 3$ are never complete Nevanlinna-Pick spaces.

\end{abstract}

\maketitle

\newsection{Introduction}

Let $\Omega$ be a non-empty set, and let $K:\Omega \times \Omega \to \C$ be a \emph{kernel} on $\Omega$ (written as $K\succeq 0$), that is, 
for any finite set $\{x_1,\ldots,x_n\}\subset \Omega$, the matrix  \[\begin{bmatrix}
K(x_i,x_j)
\end{bmatrix}_{i,j=1}^n\] is positive semi-definite. We say $K$ is a \emph{holomorphic kernel} if $\Omega$ is a domain and $K$ is holomorphic in one variable and conjugate holomorphic in the other variable.   
The reproducing kernel Hilbert space (RKHS) corresponding to $K$, denoted by $\clh(K)$, is a Hilbert function space on $\Omega$ canonically associated to $K$ (see \cite{PaulRaghu}). For kernels $K$ and $\widetilde{K}$ on $\Omega$, we say $\clh(K)$ and $\clh(\widetilde{K})$ are \emph{isometrically isomorphic as reproducing kernel Hilbert spaces} if there exist a bijective map $F:\Omega\to \Omega$ and non-vanishing function $\lambda:\Omega\to \mathbb C$ such that 
\[
K(x,y)=\lambda(x)\widetilde{K}(F(x),F(y))\overline{\lambda(y)}\quad (x,y\in\Omega).
\]
A function $\varphi:\Omega \to \C$ is a \emph{multiplier} on $\clh(K)$ if $\varphi f\in \clh(K)$ for all $f\in \clh(K)$. The set of all multipliers on $\clh(K)$ is denoted by $\clm(\clh(K))$. By an application of closed graph theorem, for each $\varphi\in\clm(\clh(K))$ the multiplication operator 
\[M_\varphi:\clh(K)\to \clh(K),\ f\mapsto \varphi f,
\] 
is a bounded operator on $\clh(K)$. A multiplier $\varphi\in \clm(\clh(K))$ is contractive if $\|M_{\varphi}\|\le 1$, equivalently if $(1-\varphi(z)\overline{\varphi(w)})K(z,w)$ is a kernel on $\Omega$. 
The collection of all contractive multipliers on $\clh(K)$ is denoted by $\clm_1(\clh(K))$.
By means of these notations, the Nevanlinna-Pick interpolation problem asks, given  a finite set $\{x_1, \ldots, x_k\}\subset \Omega$ and $w_1\ldots, w_k \in \C$, whether there exists a multiplier $\varphi$ in $\clm_1(\clh(K))$ such that
\[\varphi(x_i)=w_i \ (i=1,\ldots,k).\]
The contractivity requirement of $M_{\varphi}$ necessitates the matrix, known as Pick matrix,
\[\begin{bmatrix}
(1-w_i \overline{w_j})K(x_i, x_j)
\end{bmatrix}_{i,j=1}^{k}\]
being positive semi-definite. Therefore, the positive semi-definiteness of the Pick matrix is a necessary condition to solve the interpolation problem (see \cite[Theorem 5.2]{AglerM}). If the necessary condition is also sufficient then we say $K$ is a Nevanlinna-Pick kernel. The pioneer work of Pick (\cite{Pick}) shows that the Szeg\" o kernel on the unit disc $\D$, defined by 
\[
\cls_1(z, w):=\frac{1}{1-z \Bar{w}}, \quad (z,w\in\D)
\]
is a Nevanlinna-Pick kernel. Unaware of Pick's work, Nevanlinna also solved and parameterized all the solutions of an interpolation problem in this setting (\cite{Nevanlinna1}). In fact, the work of Pick and Nevanlinna is the beginning of this interpolation problem. In $1967$, the influential work of Sarason (\cite{Sara}) established a connection between the Nevanlinna-Pick interpolation problem and operator theory, which paved a way for further generalizations for various other domains. There are numerous generalizations starting from ~\cite{AM} for multiply connected domains, ~\cite{AglerM, BallT} for the bidisc, \cite{AglerY, BhattSau} for the symmetrized bidisc to \cite{JurryK} for distinguished varieties of the bidisc.

The matrix-valued Nevanlinna-Pick interpolation problem asks, given  a finite set of points $\{x_1, \ldots, x_k\}\subset \Omega$ and $W_1,\ldots, W_k \in M_{m,n}$, the set of all $m\times n$ matrices, whether there exists a multiplier $\varphi$ in 
\[\clm(\clh(K)\otimes \C^n,\clh(K)\otimes \C^m ):=\{f:\Omega \to M_{m,n}:f g\in \clh(K)\otimes \C^m \text{ for all }g\in\clh(K)\otimes \C^n\}\] such that
\[\varphi(x_i)=W_i \ (i=1,\ldots,k)\ \text{ and }\quad \|M_\varphi\|\leq 1.\]
Once again a necessary condition to solve the above matrix-valued interpolation problem is that the block matrix 
\[\begin{bmatrix}
(I-W_i W^*_j)K(x_i,x_j)
\end{bmatrix}_{i,j=1}^{k}\]
is positive semi-definite (see \cite[Theorem 5.8]{AglerM}). We say $K$ has the $M_{m,n}$-Nevanlinna-Pick property if the necessary condition is also sufficient. A kernel $K$ is a \emph{complete Nevanlinna-Pick kernel}, hereby abbreviated as CNP kernel, if $K$ has the $M_{m,n}$-Nevanlinna-Pick property for all $m,n\in \mathbb N$. If $K$ is a CNP kernel then the corresponding reproducing kernel Hilbert space $\clh(K)$ is called a CNP space. Prototype examples of CNP spaces are the Hardy space $H^2(\D)$ over the unit disc with Szeg\"o kernel $\cls_1$ and the Drury-Arveson space $H^2_n(\mathbb B^n)$ over the unit ball $\mathbb B^n$ in $\mathbb C^n$ with kernel
\begin{align}\label{Drury-Arveson_Kernel}
    \clk(\z,\w)=\frac{1}{1-\langle \z,\w\rangle}_{\C^n}\quad \big(\z=(z_1,\ldots,z_n), \w=(w_1,\ldots, w_n)\in \mathbb B^n\big).
\end{align} 
CNP spaces are of constant interest as they behave nicely compared to general reproducing kernel Hilbert spaces in the sense that, many important properties of the Hardy space or the Drury-Arveson space hold for CNP spaces.
To name a few, any function in a CNP space can be decomposed into sub-inner and free-outer factors which is a generalization of the classical inner-outer factorization of functions in Hardy space (\cite{JuryMartin, AHMR}); Beurling-Lax-Halmos type theorem holds in a CNP space (\cite{McTr00}); 
the Gleason problem is solvable for the Multiplier algebra of a CNP space (\cite{Hartz_1, ABJK}); last but not the least, every CNP space satisfies the column-row property (with constant $1$) (\cite{Hartz}).
As a matter of course, it is important to characterize reproducing kernel Hilbert spaces which are CNP spaces.

A characterization of CNP kernels is first studied by McCullough (\cite{Mccull92}) and Quiggin (\cite{Quiggin}). However, for the present purpose we use the following characterization by Agler and McCarthy (\cite{AglerM}).
We say a kernel $K$ on $\Omega$ is \emph{normalized} at some point $\sfw\in \Omega$ if $K(x,\sfw)=1$ for all $x\in \Omega$. The hypothesis of the kernel being normalized in the following characterization, as well as in all other results where it appears in this article, is only for convenience and does not put any restriction.  
 For Hilbert spaces $\cle$ and $\clf$, we denote by $\clb(\clf, \cle)$ the space of all bounded linear operators from $\clf$ to $\cle$. We simply write $\clb(\cle)$ to denote the space $\clb(\cle,\cle)$. The open unit ball in $\clb(\clf, \cle)$ is denoted by $\clb_1(\clf, \cle)$.

\begin{Theorem}[McCullough--Quiggin, Agler--McCarthy]\label{MQAM}
Let $K$ be a non-vanishing kernel on $\Omega$ that is normalized at some point $\sfw\in \Omega$. Then
$K$ is a CNP kernel if and only if there exist an auxiliary Hilbert space $\cle$ and a function $u:\Omega \to \clb_1(\cle, \C)$ such that $u(\sfw)=0$ and 
    \begin{align}\label{CNP_Kernel_form}
        K(x,y)=\frac{1}{1-u(x)u(y)^*}\quad (x,y\in \Omega).
    \end{align}
\end{Theorem}

In the above characterization, $K$ is a holomorphic kernel if and only if $u$ is a holomorphic function on $\Omega$.
It is well-known (\cite{AglerM}) that $K$
is a kernel on $\Omega$ if and only if there exist an auxiliary Hilbert space $\cle$ and a function $g:\Omega \to \clb(\cle, \C)$ such that $K(x,y)=g(x)g(y)^*$ for all $x,y\in \Omega$. Therefore, the above theorem gives the following equivalent characterization:
\begin{align}
    K \text{ is a non-vanishing and normalized CNP kernel if and only if } 1-\frac{1}{K}\succeq 0.
    \end{align}
Recently, Chu (\cite{Chu}) considered \textit{de Branges-Rovnyak spaces} (\cite{BranRov66, BranRov68}) of the Hardy space $H^2(\D)$ and found a characterization for them to be CNP spaces. A similar characterization in the setting of Drury-Arveson space was also obtained by Sautel in \cite{Jesse}. 
The aim of this article is to advance the study carried out in ~\cite{Chu} and ~\cite{Jesse}.

The notion of \textit{de Branges-Rovnyak subspaces} of $H^2(\D)$ was first introduced by de Branges and Rovnyak (\cite{BranRov66} and \cite{BranRov68}) in the context of model theory for a large class of contractions. Subsequently, many applications on different topics in complex analysis and operator theory were found. For more details on these, we refer the reader to \cite{BranRov66} and \cite{Sarason}. For a kernel $K$ on $\Omega$ and $\varphi\in \clm_1(\clh(K))$, the de Branges-Rovnyak space associated to $\varphi$ is a reproducing kernel Hilbert space on $\Omega$ with kernel 
\begin{align}\label{Mul_equivalent}
    K^\varphi(x,y)=(1-\varphi(x) \overline{\varphi(y)})K(x,y)\quad (x,y\in \Omega).
\end{align}
Throughout the article, for $\varphi\in\clm_1(\clh(K))$, the notations $K^{\varphi}$ and $\clh(K^{\varphi})$ are reserved for the de Branges-Rovnyak kernel and the de Branges-Rovnyak subspace of $\clh(K)$ associated to $\varphi$, respectively.
With these notations, Chu in ~\cite{Chu} proved the following remarkable and very neat characterization result for de Branges-Rovnyak subspaces of $H^2(\D)$ to be CNP spaces. Note that the multiplier algebra $\clm(H^2(\D))$ is isometrically isomorphic to $H^{\infty}(\D)$, the space of all bounded holomorphic functions on $\D$. We denote by $H_1^{\infty}(\D)$ the closed unit ball in $H^{\infty}(\D)$. 
 The following theorem is stated with the normalization that $\varphi(0)=0$. The only reason behind such a normalization is that the theorem takes much simpler form with it. The present reformulation is equivalent to \cite[Theorem 1.1]{Chu}, for more details the reader is referred to \cite[page 27]{Jesse}. 
\begin{Theorem}[Chu]\label{CChu}
Let $\varphi$ be a non-constant function in $H_1^\infty(\D)$ with $\varphi(0)=0$. Then the de Branges-Rovnyak space $\clh(\cls_1^\varphi)$ is a CNP space if and only if there exists a $\psi\in H_1^\infty(\D)$ such that 
\[\varphi(z)=z\psi(\varphi(z))\quad (z\in \D).\]
\end{Theorem}

It is apparent that the collection of bounded holomorphic functions $\varphi\in H^{\infty}_1(\D)$ for which $\clh(\cls_1^\varphi)$ is a CNP space is significantly smaller as, at the very least, $\varphi$ needs to be an injective function. Indeed, for different $z_1$ and $ z_2$ in $\D$ if $\varphi(z_1)=\varphi(z_2)$, then it follows from the identity $\varphi(z)=z\psi(\varphi(z))$ that $\varphi(z_1)=\varphi(z_2)=0$ and $\psi(0)=0$. Since $\psi(0)=0$, the functions $\varphi(z)$ and $z\psi(\varphi(z))$ has a zero at the origin with different multiplicities, which is a contradiction. Just to demonstrate how useful the above theorem can be, let us consider the examples of local Dirichlet spaces corresponding to atomic measures on the unit circle. It is known that local Dirichlet spaces can be realized as de Branges-Rovnyak spaces of $H^2(\D)$ (see \cite[Theorem 3.1]{CGR}), and by a non-trivial result of Shimorin (\cite{Shimo}), local Dirichlet spaces are also CNP spaces. However, without much of a difficulty, the above theorem can be applied to see that local Dirichlet spaces are CNP spaces (see Example~\ref{Exam_Dirich_CNP} below).

An analogous characterization in the setting of the Drury-Arveson space is also obtained in ~\cite{Jesse}. One would therefore expect that such a result should also be true for CNP spaces. This is indeed the case (see Theorem~\ref{CNP_gene_kernel} below). 
Then one wonders about the validity of such a result beyond CNP spaces. It is therefore natural to ask the following question. 

\textbf{Question:} \textit{When is a de Branges-Rovnyak space of a reproducing kernel Hilbert space a CNP space?}

This problem is very difficult to solve in its full generality, and one can not expect to find a characterization which holds true for every reproducing kernel Hilbert space. However, we answer this question for a fairly large class of reproducing kernel Hilbert spaces containing CNP spaces. Our class is motivated by the characterization of CNP kernels, that is, a non-vanishing normalized kernel $K$ is a CNP kernel if and only if $1-\frac{1}{K}\succeq 0$. In order to go beyond CNP kernels, it is therefore natural to consider non-vanishing normalized kernels $K$ such that 
\begin{align}\label{Ker_decompo}
   1- \frac{1}{K}=K_1-K_2, 
\end{align}
for some non-zero kernel $K_1$ and for some kernel $K_2$. This class of kernels obviously contains CNP spaces, as we may choose $K_2=0$, as well as many spaces which are not a CNP space. The primary members of the class, from our perspective, are the Bergman space over the unit disc with kernel 
\[B(z,w)=\frac{1}{(1-z\overline{w})^2}\quad (z,w \in \D),\]
and the Hardy space over the polydisc $\D^n$ with kernel 
\begin{equation}\label{szego kernel}
\mathcal S_n(\z, \w)=\frac{1}{\prod^n_{i=1}(1-z_i \overline{w_i})}\quad (\z=(z_1, \ldots, z_n), \w=(w_1,\ldots, w_n)\in\D^n).
\end{equation}
More generally, tensor products or Schur products (defined below) of CNP kernels are also members of the class. 
One of the main results of this article is the following characterization for the significantly large family of reproducing kernel Hilbert spaces as described above.
\begin{Theorem}\label{NCNP1}
Let $K$ be a non-vanishing kernel on $\Omega$ such that $K$ is normalized at $\sfw\in \Omega$ and  
\[1-\frac{1}{K(x,y)}=g(x)g(y)^*-f(x)f(y)^*\quad  (x,y\in \Omega)
\]
for some function $f: \Omega\to \clb(\cle, \C)$ and a non-zero function $g:\Omega\to \clb(\clf,\C)$, where $\cle$ and $\clf$ are Hilbert spaces. Suppose that $\varphi: \Omega \to \D $ with $\varphi(\sfw)=0$ is a non-constant function in $\clm_1(\clh(K))$. Then the de Branges-Rovnyak space $\clh(K^\varphi)$ is a CNP space if and only if there exists $\Psi\in H^\infty_1(\D, \clb((\cle\oplus \C), \clf))$
such that 
    \[f(x)=g(x) \psi_1(\varphi(x))\ \text{ and } \varphi(x)=g(x) \psi_2(\varphi(x)) \quad (x\in \Omega),\]
 where for all $z\in\D$, $\Psi(z)= \begin{bmatrix} \psi_1(z) & \psi_2(z)\end{bmatrix}$, 
 $\psi_1(z)\in \clb(\cle,\clf)$ and $\psi_2(z)\in \clb(\C,\clf)$.
\end{Theorem}
This theorem is proved in Section~\ref{Sec2} as Theorem~\ref{NCNP}. As a particular case, namely by taking $f=0$ in Theorem~\ref{NCNP1}, we obtain a characterization in the setting of CNP kernels which recovers the result of ~\cite{Chu} and ~\cite{Jesse} 
(see Theorem~\ref{CNP_gene_kernel} below).
We should mention here that
the normalization of $\varphi$, that is, $\varphi(\sfw)=0$, is harmless and does not put any restriction on the theorem. This is explained in details in Section ~\ref{Sec1} (see Proposition~\ref{Iso_CNPs}). The only reason behind such a normalization is that our characterization takes its simplest form with it.

We then apply Theorem~\ref{NCNP1} to kernels which are either tensor products of CNP kernels or Schur product of two CNP kernels and obtain a complete classification 
(see Theorem~\ref{NCNPm} and Theorem~\ref{NCNPs}). 
In particular, we have completely answered the following questions. 

\textbf{Question:}
(i) \textit{In the Hardy space over the polydisc, which de Branges-Rovnyak subspaces are of CNP type?}

(ii)\textit{ In the weighted Bergman spaces over the unit disc, which de Branges-Rovnyak subspaces are of CNP type?} 

The Hardy space over $\D^n$ is denoted by $H^2(\D^n)$ and the kernel for $H^2(\D^n)$ is the Szeg\H o kernel $\cls_n$ as in ~\eqref{szego kernel}. It turns out that the answer to question (i) above depends on $n$. \textit{For $n\ge 3$, we show that there is no non-trivial de Branges-Rovnyak space of $H^2(\D^n)$ which is a CNP space}. However, for $H^2(\D^2)$ note that if we take $\varphi(\z)=z_i$ $(i=1,2)$ then the de Branges-Rovnyak space $\clh(\cls_2^{\varphi})$ is isometrically isomorphic to $H^2(\D)$ as reproducing kernel Hilbert spaces. Therefore it is a CNP space. 
We show that these are all the de Branges-Rovnyak spaces which are CNP spaces, in some sense. More precisely, we prove the following theorem. We denote by $b_{\mu}$ the automorphism of $\D$ corresponding to $\mu\in \D$ and by $\mathbb{T}$ the unit circle, that is, $\mathbb{T}:=\{z\in \mathbb{C}: |z|=1\}$.
\begin{Theorem}\label{Hardy space over the bidisc}
Let $\varphi\in H^{\infty}_1(\D^2)$. Then the de Branges-Rovnyak space $\clh(\cls_2^{\varphi})$ is a CNP space if and only if $\varphi(\z)=\lambda b_{\mu}(z_i)$ for some $\mu\in\D$, $\lambda\in \mathbb{T}$, and $i=1,2$.
\end{Theorem}
This result is obtained as a byproduct of a general characterization result in the setting of tensor product of two CNP kernels. 
See Theorem~\ref{NCNPm} below for more details. Thus only a handful of de Branges-Rovnyak subspaces of $H^2(\D^2)$ are CNP spaces. This is perhaps expected as even model spaces are known to have very complicated structure.

The weighted Bergman space over the unit disc has the kernel $B_{\alpha}(z,w)=\frac{1}{(1-z\bar{w})^{2+\alpha}}$ $(z,w\in\D, \alpha> -1)$. For $\alpha=0$, the kernel $B_0$, also denoted as $B$, is the kernel for the Bergman space over the unit disc. It is easy to see that for $\varphi(z)=z$, the de Branges-Rovnyak space $\clh(B^{\varphi})$ has the kernel $\frac{1}{1-z\bar{w}}$, and hence it is a CNP space. We show that for  any $\varphi\in H^{\infty}_1(\D)$, the de Branges-Rovnyak space $\clh(B^{\varphi})$ of the Bergman space is a CNP space if and only if $\varphi(z)=\lambda b_{\mu}(z)$ for some $\lambda\in \mathbb{T}$ and $\mu\in\D$. This result is obtained as a consequence of a general characterization in the setting of Schur product of CNP kernels (see Theorem~\ref{NCNPs} below for more details). The authors came to know after completing the work that this result is also obtained in a recent article ~\cite{LuoZhu}. We also show that for the weighted Bergman spaces with kernel $B_{\alpha}$ $(\alpha\ge 1)$, there are no non-trivial de Branges-Rovnyak spaces which are CNP. This partially answers a question left open in ~\cite{LuoZhu} (see page $2$ in \cite{LuoZhu}), as our approach does not seem to work for the case $0<\alpha<1$.

The article is organized as follows. In the next section, we state some known results and show that the normalization we assume in various results is harmless. Our characterization for de Branges-Rovnyak spaces of a large class of non-CNP kernels which are CNP spaces is considered in Section~\ref{Sec2}. Tensor products of CNP kernels and Schur product of two CNP kernels are considered in Section~\ref{Sec 3}, and as a consequence the results for the weighted Bergman spaces on $\D$ and the Hardy space on $\D^n$ are obtained.

\newsection{Background materials}\label{Sec1}

For a scalar-valued kernel $K$ on $\Omega$ and a Hilbert space $\cle$, the space $\clh(K)\otimes\cle$ is a reproducing kernel Hilbert space with operator-valued kernel $(x,y)\in \Omega \times \Omega \mapsto K(x,y) I_\cle$. An element $f$ in $\clh(K)\otimes\cle$ is viewed as an $\cle$-values function $f:\Omega\to \cle$ and it satisfies
\[\langle f(x), \eta\rangle_{\cle}=\langle f, K(., x)\otimes\eta\rangle_{\clh(K)\otimes\cle} \quad (x\in \Omega, \eta\in \cle).\]
For Hilbert spaces $\cle$ and $\clf$, a function $\Theta :\Omega \to \clb(\clf,\cle)$ is said to be a multiplier from $\clh(K)\otimes \clf$ to $\clh(K)\otimes \cle$ if $\Theta f\in \clh(K)\otimes \cle$ for all $f\in \clh(K)\otimes \clf$, where $(\Theta f)(x)=\Theta(x)(f(x))$ for all $x\in\Omega$. By an application of closed graph theorem, it is easy to show that the multiplication operator 
\[M_\Theta : \clh(K)\otimes \clf \to \clh(K)\otimes \cle,\ f\mapsto \Theta f,\] is a bounded operator.  
The space of multipliers from $\clh(K)\otimes \clf$ to $\clh(K)\otimes \cle$ with multiplier norm is a Banach space which we denote by $\clm(\clh(K)\otimes \clf,\clh(K)\otimes \cle)$ and the closed unit ball is denoted by $\clm_1(\clh(K)\otimes \clf,\clh(K)\otimes \cle)$. For more details on the theory of vector-valued reproducing kernel Hilbert spaces and their multipliers we refer the reader to \cite{PaulRaghu, AglerM}. We make crucial use of two well-known results in the theory of CNP spaces. One of them is a Douglas type factorization result due to Leech (\cite{Leech}); for the following version of the result see ~\cite[Theorem 8.57]{AglerM}. 

\begin{Theorem}[Leech]\label{Leech_Theo}
Let $K$ be a CNP kernel on $\Omega$. Suppose $g:\Omega \to \clb(\cle, \C)$ and $h:\Omega \to \clb(\clf, \C)$ are given functions for some Hilbert spaces $\cle$ and $\clf$. Then 
\[\big(g(x) g(y)^*-h(x)h(y)^*\big)K(x,y)\succeq 0\]
if and only if  there exists a multiplier $\Theta \in \clm_1(\clh(K)\otimes \clf,\clh(K)\otimes \cle)$   such that 
\[h(x)=g(x)\Theta(x) \quad (x\in \Omega).\]
\end{Theorem}
The other result that we need is due to Ball, Trent, and Vinnikov (\cite[Theorem 3.1]{Ball}), which describes the space of multipliers on CNP spaces. Some more terminologies are needed to state the result. Let $\cle$ be a Hilbert space, and let $\mathbb B_1(\cle)$ be the open unit ball in $\cle$. Then the Drury-Arveson space on $\mathbb B_1(\cle)$ is a reproducing kernel Hilbert space with kernel
\begin{equation}\label{DA kernel}
\mathcal K(\zeta,\eta)=\frac{1}{1-\langle \zeta, \eta\rangle}\quad (\zeta,\eta \in \mathbb B_1(\cle)).
\end{equation}

For each function $u:\Omega\to \clb_1(\cle,\mathbb C)$, we denote by $K_u$ the CNP kernel associated to $u$ defined as 
\begin{equation}\label{CNP kernel}
K_{u}(x,y)= \frac{1}{1-u(x)u(y)^*}\quad (x,y\in\Omega).
\end{equation}
Recall that, by Theorem~\ref{MQAM}, any non-vanishing CNP kernel $K$ on $\Omega$ which is normalized at $\sfw\in\Omega$ is of the form $K_u$ for some $u:\Omega\to \clb_1(\cle,\mathbb C)$ with $u(\sfw)=0$. In addition, $K$ is a holomorphic kernel if and only if $u$ is holomorphic on $\Omega$.

\begin{Theorem}[Ball, Trent, and Vinnikov]\label{Theo_BTV}
  Let $K_u$ be a CNP kernel on $\Omega$ as in \eqref{CNP kernel}. 
  Then $\Theta \in \clm_1(\clh(K_u)\otimes \clf,\clh(K_u)\otimes \clg)$ for some Hilbert spaces $\clf$ and $\clg$ if and only if there exists a $\Psi\in \clm_1(\clh(\clk)\otimes \clf, \clh(\clk)\otimes \clg)$ such that  
  \[\Theta(x)=\Psi(u(x)),\quad (x\in \Omega)\]
where $\clk$ is the kernel of the Drury-Arveson space on $\mathbb B_1(\cle)$ as in ~\eqref{DA kernel}. 
\end{Theorem}

 In what follows, we consider non-vanishing kernels on $\Omega$ which are normalized at $\sfw\in\Omega$. The normalization of the kernel is not essential. Nevertheless, we make this assumption for simplicity. For such a kernel $K$, the central object  of our study is de Branges-Rovnyak subspaces of $\clh(K)$ corresponding to non-constant functions $\varphi:\Omega\to \D$ in $\mathcal M_1(\clh(K))$ such that $\varphi(\sfw)=0$. The normalization that $\varphi(\sfw)=0$ does not put any restriction in our characterization. Because corresponding to any non-constant multiplier $\varphi:\Omega\to \D$ one can construct a multiplier $\tilde{\varphi}:\Omega \to \D$ such that $\tilde{\varphi}(\sfw)=0$, and the de Branges-Rovnyak spaces corresponding to $\varphi$ and $\tilde{\varphi}$ are isometrically isomorphic as reproducing kernel Hilbert spaces. This is equivalent to M\"obius invariance of de Branges-Rovnyak kernels. Such a result is well known (See \cite[Lemma 4.7]{LuoGR}). However, we include a proof for completeness of this article. We begin with a simple lemma.

  \begin{Lemma} \label{Mul_unitary}
Let $K$ and $\widetilde{K}$ be two kernels on $\Omega$. If $f\in\clm(\clh(K),\clh(\widetilde{K}))$
is a nowhere vanishing function and $M_f^*$ is an isometry, then $M_f$ is a unitary. 
    \end{Lemma}
    \begin{proof}
Given any point $\lambda\in\Omega$, the kernel vector at $\lambda$ corresponding to $K$ is denoted by $K(.,\lambda)$. It is well-known that for any $f\in\clm(\clh(K),\clh(\widetilde{K}))$, $M_f^*(\widetilde{K}(.,\lambda))=\overline{f(\lambda)}K(.,\lambda)$. Then the proof follows from the fact that kernel vectors form a total subset of the corresponding reproducing kernel Hilbert space. 
\end{proof}

\begin{Proposition}\label{Iso_CNPs}
Let $K$ be a kernel on $\Omega$ and $\sfw\in\Omega$. Then for each function $\varphi:\Omega\to \D$ in $\clm_1(\clh(K))$ there exists a function $\tilde{\varphi}:\Omega\to \D$ in $\clm_1(\clh(K))$ with  $\Tilde{\varphi}(\sfw)=0$ such that $\clh(K^\varphi)$ is isometrically isomorphic to $\clh(K^{\Tilde{\varphi}})$ as reproducing kernel Hilbert spaces.
Moreover, $\clh(K^\varphi)$ is a CNP space if and only if $\clh(K^{\Tilde{\varphi}})$ is a CNP space.
\end{Proposition}
\begin{proof}
If $\varphi (\sfw)=0$, then there is nothing to prove. Let us assume that $\varphi(\sfw)=\mu\neq 0$. Set $\Tilde{\varphi}:=b_{\mu}\circ\varphi$, where  $b_{\mu}$ is the disc automorphism defined by $b_{\mu}(z)=\frac{z-\mu}{1-\bar{\mu}z}$ ($z\in\D$). Then $\tilde{\varphi}(\sfw)=0$ and for all $x,y\in\Omega$,
\begin{align}\label{CNP_Identity}
K^{\Tilde{\varphi}}(x,y)&=(1-\Tilde{\varphi}(x)\overline{\Tilde{\varphi}(y)})K(x, y)\nonumber\\
& = \left(1-\frac{(\varphi(x)-\mu) (\overline{\varphi(y)}-\bar{\mu})}{(1-\bar{\mu}\varphi(x))(1-\mu\overline{\varphi(y)})}\right) K(x,y)\nonumber\\
&=\frac{(1-|\mu|^2)}{(1-\bar{\mu}\varphi(x))(1-\mu\overline{\varphi(y)})}(1-\varphi(x)\overline{\varphi(y)})K(x,y)\nonumber\\
&= f(x)K^\varphi(x,y) \overline{f(y)},
    \end{align}
where $f(x)=\frac{\sqrt{(1-|\mu|^2)}}{1-\bar{\mu}\varphi(x)}\neq 0$ for all $x\in\Omega$. From the above identity one can infer that $\tilde{\varphi}\in \clm_1(\clh(K))$, $f\in\clm(\clh(K^\varphi),\clh(K^{\Tilde{\varphi}}))$, and 
\[M_fM_f^*(K^{\Tilde{\varphi}}(., y)) = K^{\Tilde{\varphi}}(., y)
\]
for all $y\in\Omega$. Thus $M_{f}^*$ is an isometry, and hence a unitary by Lemma ~\ref{Mul_unitary}. Since the function $f(x)=\frac{\sqrt{(1-|\mu|^2)}}{1-\bar{\mu}\varphi(x)}\neq 0$ for all $x\in \Omega$, the moreover part follows by applying \cite[Theorem 7.28]{AglerM} to the identity \eqref{CNP_Identity}. This completes the proof.
\end{proof}

\textit{We use the following notation throughout the article:}  
\[H^\infty_1(\D, \clb(\clf, \clg)):=\{\varphi:\D \to \clb(\clf, \clg)|\ \varphi \text{ is holomorphic on } \D, \|\varphi\|_{\infty}=\sup_{z\in \D}\|\varphi(z)\|\leq 1\}.\]

\newsection{Characterization: The CNP and non-CNP case}\label{Sec2}
The main content of this section is the proof of Theorem~\ref{NCNP1}. We again state Theorem~\ref{NCNP1} below for the convenience of the reader. 

\begin{Theorem}\label{NCNP}
Let $K$ be a non-vanishing kernel on $\Omega$ such that $K$ is normalized at $\sfw\in\Omega$ and  
\[1-\frac{1}{K(x,y)}=g(x)g(y)^*-f(x)f(y)^*\quad  (x,y\in \Omega)
\]
for some function $f: \Omega\to \clb(\cle, \C)$ and a non-zero function $g:\Omega\to \clb(\clf,\C)$, where $\cle$ and $\clf$ are Hilbert spaces. Suppose that $\varphi: \Omega \to \D $ with $\varphi(\sfw)=0$ is a non-constant function in $\clm_1(\clh(K))$. Then the de Branges-Rovnyak space $\clh(K^\varphi)$ is a CNP space if and only if there exists $\Psi\in H^\infty_1(\D, \clb((\cle\oplus \C), \clf))$
such that 
    \[f(x)=g(x) \psi_1(\varphi(x))\ \text{ and } \varphi(x)=g(x) \psi_2(\varphi(x)) \quad (x\in \Omega),\]
 where for all $z\in\D$, $\Psi(z)= \begin{bmatrix} \psi_1(z) & \psi_2(z)\end{bmatrix}$, 
 $\psi_1(z)\in \clb(\cle,\clf)$ and $\psi_2(z)\in \clb(\C,\clf)$.
\end{Theorem}
\begin{proof}
Since $K$ is a non-vanishing and normalized at $\sfw$, the de Branges-Rovnyak kernel $K^\varphi(x,y)=(1-\varphi(x)\overline{\varphi(y)})K(x,y)$ ($x,y\in \Omega$) is non-vanishing and normalized at $\sfw$.  Then, by Theorem~\ref{MQAM}, $K^{\varphi}$ is a CNP kernel if and only if $1-\frac{1}{K^{\varphi}}\succeq 0$.
By the hypothesis 
\[
1-\frac{1}{K(x,y)}=g(x)g(y)^*-f(x)f(y)^*,\quad (x,y\in \Omega)
\]
for some function $f: \Omega\to \clb(\cle, \C)$ and a non-zero function $g:\Omega\to \clb(\clf,\C)$.
Then for all $x,y\in \Omega$,
\begin{align*}
    1-\frac{1}{K^\varphi(x,y)}&=1-\frac{1}{\big(1-\varphi(x)\overline{\varphi(y)}\big)K(x,y)}\\
    &= 1-\frac{1-g(x)g(y)^* +f(x)f(y)^*}{1-\varphi(x)\overline{\varphi(y)}}\\
    &= \frac{g(x)g(y)^*- (\varphi(x)\overline{\varphi(y)}+f(x)f(y)^*)}{1-\varphi(x)\overline{\varphi(y)}}\\
    &=(g(x)g(y)^*-h(x)h(y)^*)K_{\varphi}(x,y)
\end{align*}
where $h:\Omega\to \clb((\cle\oplus \C), \C)$ is defined by $h(x)=\begin{bmatrix} f(x) & \varphi(x) \end{bmatrix}$ for all $x\in\Omega$. Consequently, by Theorem~\ref{Leech_Theo}, $K^{\varphi}$ is a CNP kernel if and only if there exists $\Theta\in\clm_1(\clh(K_\varphi)\otimes (\cle \oplus \C),\clh(K_{\varphi})\otimes \clf)$ such that 
\begin{align}\label{Mul_eq0}
    h(x)=g(x) \Theta(x) \quad (x\in \Omega).
\end{align}
Since $\Theta\in\clm_1(\clh(K_\varphi)\otimes (\cle \oplus \C),\clh(K_\varphi)\otimes \clf)$, by Theorem \ref{Theo_BTV}, there exists $\Psi\in H^\infty_1(\D, \clb((\cle\oplus \C), \clf))$ such that 
\[\Theta(x)=\Psi(\varphi(x)) \quad (x\in \Omega).\]
Let $\Psi(z)=\begin{bmatrix} \psi_1(z) & \psi_2(z)\end{bmatrix}$ for all $z\in\D$. Then $\psi_1\in H^\infty_1(\D, \clb(\cle, \clf))$, $\psi_2\in H^\infty_1(\D, \clb(\C, \clf))$ and the identity ~\eqref{Mul_eq0} yields
\[\begin{bmatrix} f(x) & \varphi(x)\end{bmatrix}=g(x)\begin{bmatrix} \psi_1 (\varphi(x)) & \psi_2(\varphi(x))\end{bmatrix} \quad (x\in \Omega).\]  
The proof now follows by comparing entries of the above block operator matrix identity.
\end{proof}
 
If we take $f=0$ in Theorem~\ref{NCNP}, then $K$ becomes a CNP kernel, and in such a case $\psi_1=0$. Thus it provides the following characterization for de Branges-Rovnyak subspaces of a CNP space to be CNP spaces, which extends the result of Chu obtained in ~\cite{Chu}.   

\begin{Theorem}\label{CNP_gene_kernel} 
Let $K$ be a CNP kernel on $\Omega$ that is normalized at $\sfw\in\Omega$, and let $u:\Omega\to \clb_1(\cle,\C)$ be a function such that $u(\sfw)=0$ and $K=K_u$ according to Theorem~\ref{MQAM}.
Let $\varphi:\Omega\to \D$ be a non-constant function in $\clm_1(\clh(K_u))$ such that $\varphi(\sfw)=0$. Then the de Branges-Rovnyak space $\clh(K_u^\varphi)$ is CNP if and only if there exists a $\Psi\in H^\infty_1(\D,\clb(\C, \cle))$ such that
    \begin{align}\label{cnp_de_cnp}
        \varphi(x)=u(x)\Psi(\varphi(x))\quad (x \in \Omega).
    \end{align}
\end{Theorem}

\begin{Remark}
The above characterization also holds for multipliers in $\mathcal M_1(\clh(K_u)\otimes \clf, \clh(K_u)\otimes \C)$ for some Hilbert space $\clf$. Indeed, if $\varphi\in \mathcal M_1(\clh(K_u)\otimes \clf, \clh(K_u)\otimes \C)$ with $\varphi(\sfw)=0$ then one can show by the same way that $\clh(K^{\varphi}_u)$ is a CNP space if and only if there exists a $\Psi\in \clm_1(\clh(\clk)\otimes \clf, \clh(\clk)\otimes \cle)$ such that 
\[
\varphi(x)= u(x)\Psi(\varphi(x))\quad (x\in\Omega),
\]
where $\clk$ is the kernel of the Drury-Arveson space on $\mathbb B_1(\clf)$ as in ~\eqref{DA kernel}.
\end{Remark}

We end this section by considering a class of examples for both CNP kernels and non CNP kernels. We begin with the case of CNP kernels. 
\begin{Example}\label{Exam_Dirich_CNP}
Let $u:\Omega\to \D$ be a function on $\Omega$ such that $u(\sfw)=0$ for some $\sfw\in\Omega$. Let 
\[\varphi(x)= \frac{a u(x)}{1-b u(x)} \quad (x\in\Omega)
\]
where $0\neq a,b\in\C$ and  $|a|+|b|\leq1$. Since $\frac{az}{1-bz}\in H^\infty_1(\D)$, it follows from Theorem~\ref{Theo_BTV} that $\varphi\in \mathcal M_1(K_u)$. Clearly, $\varphi(\sfw)=0$. Now consider the linear polynomial $\Psi(z)= a+ bz$ $(z\in\D)$.  Since $|a|+|b|\leq1$, $\Psi $ is in $ H^\infty_1(\D)$. Furthermore, for all $x\in \Omega$
\[
u(x)\Psi(\varphi(x))=u(x)\left(a+\frac{abu(x)}{1-bu(x)}\right )=\varphi(x).
\]
Therefore, by Theorem \ref{CNP_gene_kernel}, the de Branges-Rovnyak space $\clh(K^{\varphi}_u)$ is a CNP space. The linear polynomial $\Psi$ and $\varphi$ determine each other. Indeed, if $u(x)\Psi(\varphi(x))=\varphi(x)$ $(x\in\Omega)$ for some $\varphi\in \mathcal M_1(K_u)$ and $\Psi(z)=a+bz$ then a simple calculation shows that $\varphi(x)=\frac{au(x)}{1-b u(x)}$ for all $x\in\Omega$. 

In particular, if we take $\Omega=\D$, $\sfw=0$ and $u(z)=z$, then 
\[\varphi(z)=\frac{az}{1-bz}\quad  (z\in \D)\] where $0\neq a,b\in\C$ and  $|a|+|b|\leq1$. Therefore, the de Branges-Rovnyak subspace $\clh(K_z^\varphi)$ of $H^2(\D)$ is a CNP space. It should be noted that for particular choices of $a$ and $b$, the de Branges-Rovnyak spaces $\clh(K_z^\varphi)$ are local Dirichlet spaces (\cite{CGR,FKMR}), which are known to be CNP spaces by a result of Shimorin (\cite{Shimo}). 
\end{Example} 

We consider several examples for non-CNP kernels next.

\begin{Example}

\textup{(1)} Let $\Omega$ be a domain, and $u:\Omega\to \D$ be a holomorphic function such that $u(\sfw)=0$ for some $\sfw\in \Omega$. For $0<a<1$, consider the kernel on $\Omega$ 
 \[
 K(x,y)=\frac{1-a^2u(x)\overline{u(y)}}{(1-u(x)\overline{u(y)})^2}\quad (x,y\in\Omega).
\]
The kernel $K$ is a non-vanishing and normalized kernel on $\Omega$. One can show, using Theorem~\ref{MQAM} or invoking Theorem~\ref{NCNPs} below, that $K$ is not a CNP kernel on $\Omega$. If we take $\varphi(x)=u(x)$, then $\varphi$ is an element in $\mathcal M_1(K)$ because $au\in \clm_1(K_{u})$ (by Theorem~\ref{Theo_BTV}) and consequently,  
\[
K^{\varphi}(x,y)= \frac{1-a^2u(x)\overline{u(y)}}{1-u(x)\overline{u(y)}}\succeq 0.
\]
It can be seen from Example~\ref{Exam_Dirich_CNP}, by taking $b=0$, that the de Branges-Rovnyak kernel $K^{\varphi}$ is a CNP kernel. However, we can also use Theorem~\ref{NCNP} to prove the same. Indeed, a simple calculation reveals that 
\[
1-\frac{1}{K(x,y)}=g(x)g(y)^* - f(x)f(y)^*,
\]
where $g:\Omega\to\clb(\mathbb C,\mathbb C)$ and $f:\Omega\to\clb(\ell^2,\mathbb C)$ are given by 
\[
x\mapsto\sqrt{2-a^2}u(x) \ \text{ and } x\mapsto (1-a^2)(u(x)^2,au(x)^3,a^2u(x)^4,\ldots, a^nu(x)^{n+2},\ldots),
\]
respectively. Now we take $\Psi(z)=\begin{bmatrix} \psi_1(z) &\psi_2(z)\end{bmatrix}\in \clb(\ell^2\oplus\mathbb C,\mathbb C)$ where for all $z\in\D$ 
\[\psi_1(z)=\frac{1-a^2}{\sqrt{2-a^2}}(z, az^2, a^2z^3,\ldots, a^nz^{n+1},\ldots)\ \text{ and } \psi_2(z)=\frac{1}{\sqrt{2-a^2}}.\]
Then by a straightforward calculation, $\Psi\in H^{\infty}_1(\D, \clb(\ell^2\oplus\mathbb C,\mathbb C))$,
\[f(x)=g(x)\psi_1(\varphi(x)) \text{ and } \varphi(x)=g(x)\psi_2(\varphi(x))\quad (x\in\Omega).\] 
Thus by Theorem~\ref{NCNP} $K^\varphi$ is a CNP kernel.

\textup{(2)}
We now consider the de Branges-Rovnyak subspace of the Hardy space corresponding to 
\[\theta(z)=\sqrt{1-a^2}(b+acz)z,\quad (z\in\D)\]
where $a,b,c\in\C$ such that $0<a<1$, $c\neq 0$, $|b|+|c|<1$ and $|\frac{b}{ac}|<1$. The kernel of the de Branges-Rovnyak space corresponding to $\theta$ is given by 
\[
K(z,w)=\frac{1-\theta(z)\overline{\theta(w)}}{1-z\bar w}\quad (z,w\in \D).
\]
Since $|\frac{b}{ac}|<1$, $\theta$ has multiple roots in $\D$, and therefore, by Theorem~\ref{CChu}, $K$ is not a CNP kernel on $\D$. However, note that $K$ is a non-vanishing kernel which is normalized at $0$ and 
\[
1-\frac{1}{K(z,w)}=g(z)g(w)^*-f(z)f(w)^*,\quad (z,w\in\D)\]
where $f:\D\to \clb(\ell^2,\C)$ and $g:\D\to \clb(\ell^2,\C)$ are given by
\[f(z)=(\theta(z),\theta(z)^2,\ldots)\ \text{and }g(z)=z(1,\theta(z),\theta(z)^2,\ldots) \quad (z\in \D).
\]
Using our characterization, we show below that the de Branges-Rovnyak subspace $\clh(K^{\varphi})$ of $\clh(K)$ is a CNP space, where $\varphi(z)=az$ for all $z\in\D$. We leave it to the reader to check that $\varphi$ is a multiplier in $\clm_1(\clh(K))$. 

By Theorem \ref{NCNP}, $K^\varphi$ is a CNP kernel if and only if there exists $\Psi\in H^\infty_1(\D,\clb((\ell^2\oplus\mathbb C),\ell^2)$ such that 
\begin{equation}\label{Identity 5}f(z)=g(z)\psi_1(\varphi(z))\ \text{and}\ \varphi(z)=g(z)\psi_2(\varphi(z)),
\end{equation}
where $\Psi(z)=\begin{bmatrix} \psi_1(z) & \psi_2(z)\end{bmatrix}$ and $\psi_1(z) \in \clb(\ell^2,\ell^2)$, $\psi_2(z) \in \clb(\mathbb C,\ell^2)$ for all $z\in\D$.
A straightforward calculation shows that the identity~\eqref{Identity 5} holds with 
\[
\psi_1(z)=\begin{bmatrix}
    \varphi_1(z) & 0 & 0 & \cdots\\
    0 & \varphi_1(z) & 0 &\cdots\\
    0 & 0 & \ddots & \\
    \vdots & \vdots & 
\end{bmatrix}, \quad (z\in\D)
\]
where $\varphi_1(z)=\sqrt{1-a^2}(b+cz)$ and  
$$
\psi_2(z)=\begin{bmatrix}
    a\\0\\ \vdots
\end{bmatrix}\ \quad (z\in\D).
$$
This shows that $K^{\varphi}$ is a CNP kernel.

More generally, for any function $u:\Omega\to \D$ with $u(\sfw)=0$ for some $\sfw\in\Omega$ and 
\[\theta(x)=\sqrt{1-a^2}(b+acu(x))u(x),\quad (x\in\Omega)
\]
 where $a,b,c$ as above, one can show by the same way that $K^{\varphi}$ is a CNP kernel on $\Omega$ where 
\[
K(x,y)=\frac{1-\theta(x)\overline{\theta(y)}}{1-u(x)\overline{u(y)}}\quad (x,y\in \Omega)
\]
and $\varphi(x)=au(x)$ for all $x\in\Omega$.
\end{Example}

\section{Two Applications}\label{Sec 3}
In this section, we apply Theorem~\ref{NCNP} for tensor products of CNP kernels and Schur product of two CNP kernels. The Hardy space over polydisc and the Bergman space over the disc emerge as a particular case.

\subsection*{Tensor product of CNP kernels}
\textit{For the rest of the paper, we assume that $\Omega$ is a domain}. Let $K_1$ and $K_2$ be kernels on $\Omega$. Then the tensor product of $K_1$ and $K_2$ is a kernel on $\Omega\times \Omega$ denoted by $K_1\otimes K_2$ and defined as 
   \[
   (K_1\otimes K_2)((x_1,x_2), (y_1,y_2))= K_1(x_1,y_1)K_2(x_2,y_2)\quad (x_1,x_2,y_1,y_2\in\Omega).
   \]
   The reproducing kernel Hilbert space $\clh(K_1\otimes K_2)$ corresponding to $K_1\otimes K_2$ can be identified as the Hilbert space tensor product $\clh(K_1)\otimes \clh(K_2)$, which is perhaps the reason why we call the kernel as tensor product of kernels. For more details see ~\cite[Theorem 5.11]{PaulRaghu}. We now consider tensor product of two holomorphic CNP kernels and determine completely when a de Branges-Rovnyak subspace is a CNP space. This, in particular, provides a characterization for the Hardy space over the bidisc.

\begin{Theorem}\label{NCNPm}
    Let $u: \Omega\to \D$ and $v:\Omega\to \D$ be non-constant holomorphic functions such that $u(\sfw_1)=0=v(\sfw_2)$ for some $\mathbf w=(\sfw_1,\sfw_2)\in\Omega^2$. Consider the kernel  $K=K_{u}\otimes K_{v}$ on $\Omega \times \Omega$. Assume that $\varphi: \Omega\times\Omega\to \D$ is a non-constant function in $\clm_1(\clh(K))$ with $\varphi(\mathbf w)=0$. Then the de Branges-Rovnyak space $\clh(K^\varphi)$ is a CNP space if and only if either $\varphi(x,y)=\lambda u(x)$ or $\varphi(x,y)=\lambda v(y)$ $((x,y)\in \Omega \times \Omega)$, for some $\lambda\in \mathbb{T}$.
\end{Theorem}

\begin{proof} 
For $(x_1,x_2), (y_1,y_2)\in\Omega\times \Omega$,
\begin{align*}
    1-\frac{1}{ K((x_1, x_2),(y_1,y_2))} &=u(x_1)u(y_1)^*+ v(x_2)v(y_2)^*- u(x_1)v(x_2)u(y_1)^*v(y_2)^*\\
    & =g(x_1, x_2) g(y_1,y_2)^*-f(x_1,x_2)f(y_1,y_2)^*,
\end{align*}
where $f(x_1,x_2)= u(x_1)v(x_2)\in \clb(\C) $ and 
$g(x_1,x_2)= \begin{bmatrix} u(x_1) & v(x_2)\end{bmatrix}\in \clb(\C^2, \C)$.
Then by Theorem \ref{NCNP}, $K^{\varphi}$ is a CNP kernel if and only if there exists $\Psi\in H^\infty_1(\D, \clb(\C^2))$
such that for all $(x,y)\in\Omega\times \Omega$
\begin{align}\label{Iden2}
    u(x)v(y)=u(x) \psi_{11}(\varphi(x,y)) + v(y) \psi_{21}(\varphi(x,y))
\end{align}
and 
\begin{align}\label{Iden3}
    \varphi(x,y)=u(x) \psi_{12}(\varphi(x,y))+ v(y) \psi_{22}(\varphi(x,y)), 
\end{align}
where $\psi_{ij}\in H_1^\infty(\D)$ ($i,j=1,2$) and 
\[\Psi(z)=\begin{bmatrix} \psi_{11}(z) & \psi_{12}(z)\\
\psi_{21}(z) & \psi_{22}(z)\end{bmatrix}\] 
for all $z\in\D$.
We now show that $\varphi$ satisfies ~\eqref{Iden2} and ~\eqref{Iden3} if and only if $\varphi(x,y)=\lambda u(x)$ or $\varphi(x,y)=\lambda v(y)$ for some $\lambda\in \mathbb{T}$. For one implication, note that if $\varphi(x,y)=\lambda u(x)$, then 
\[
K^{\varphi}((x_1,x_2),(y_1,y_2))=\frac{1}{1-v(x_2)\overline{v(y_2)}}, \quad ((x_1,x_2),(y_1,y_2)\in\Omega\times \Omega)
\]
which is a CNP kernel by Theorem~\ref{MQAM}. Therefore, $\varphi$ satisfies  ~\eqref{Iden2} and ~\eqref{Iden3}. The case when $\varphi(x,y)=\lambda v(y)$ is similar. 

For the other implication, we now consider the slice functions $\varphi(.,\sfw_2)$ and $\varphi(\sfw_1, .)$ on $\Omega$ defined as 
\[\varphi(., \sfw_2 )(x)=\varphi(x,\sfw_2)\ \text{ and } \varphi(\sfw_1, .)(x)=\varphi(\sfw_1,x) \quad (x\in\Omega).
\]
\textbf{Claim:} Exactly one of the slice functions $\varphi(.,\sfw_2)$ and $\varphi(\sfw_1, .)$ is identically zero.

\textit{Proof of the claim.} First we show that both $\varphi(.,\sfw_2)$ and $\varphi(\sfw_1, .)$ can not be non-zero functions. For the sake of contradiction suppose both are non-zero. Since $u$ and $v$ are non-zero holomorphic functions on $\Omega$, from ~\eqref{Iden2} we have
\[\psi_{11}(\varphi(x,\sfw_2))=0\quad \text{and}\quad \psi_{21}(\varphi(\sfw_1,x))=0\]
for all $x\in\Omega$. Then by open mapping theorem, both $\psi_{11}$ and $\psi_{21}$ are identically zero on $\D$. 
Then by ~\eqref{Iden2}, $u(x)v(y)=0$ for all $(x,y)\in\Omega$, which is a contradiction. Now, we show that both $\varphi(.,\sfw_2)$ and $\varphi(\sfw_1, .)$ can not be identically zero. Again for the contradiction suppose that both are identically zero. Then from \eqref{Iden2} and \eqref{Iden3}, we get $\psi_{ij}(0)=0$ for all $i,j=1,2$. Consequently, by Schwarz's lemma, $\psi_{ij}(z)=z \tilde{\psi}_{ij}(z)$ $(z\in\D)$ for some $\tilde{\psi}_{ij}\in H_1^\infty(\D)$. 
Therefore, by \eqref{Iden3}, we have for all $(x,y)\in \Omega\times\Omega$,
\[\varphi(x,y)=u(x) \varphi(x,y) \tilde{\psi}_{12}(\varphi(x,y))+ v(y) \varphi(x,y) \tilde{\psi}_{22}(\varphi(x,y)).\]
Since $\varphi$ is a non-zero holomorphic function, we also have
\[u(x)  \tilde{\psi}_{12}(\varphi(x,y))+ v(y) \tilde{\psi}_{22}(\varphi(x,y)) =1 \]
for all $(x,y)\in\Omega\times \Omega$, which is a contradiction as the left hand side vanishes at $\mathbf w=(\sfw_1,\sfw_2)$. This completes the proof of the claim. 

 Let us now assume that the function $\varphi(\sfw_1,.)$ is identically zero on $\Omega$. In such a case, we show that $\varphi$ has the form $\varphi(x,y)=\lambda u(x)$ for some $\lambda\in \mathbb{T}$. Because of the above claim, $\varphi(.,\sfw_2)$ is a non-zero function on $\Omega$, and therefore by ~\eqref{Iden2} $\psi_{11}(\varphi(x,\sfw_2))=0$ for all $x\in\Omega$. This implies $\psi_{11}$ is identically zero on $\D$. By taking $x=\sfw_1$ in ~\eqref{Iden2}, we also have $\psi_{21}(0)=0$. Then $\psi_{21}(z)=z\tilde{\psi}_{21}(z)$ $(z\in\D)$ for some $\tilde{\psi}_{21}\in H^{\infty}_1(\D)$ and ~\eqref{Iden2} now becomes
\begin{align}\label{Iden4}
\varphi(x,y) \tilde{\psi}_{21}(\varphi(x, y))=u(x) 
\end{align}
for all $(x,y)\in \Omega \times \Omega$.
On the other hand, taking $y=\sfw_2$ in ~\eqref{Iden3} and using ~\eqref{Iden4}, we have
\begin{align*}
    \varphi(x,\sfw_2)=u(x)\psi_{12}(\varphi(x,\sfw_2))=\varphi(x,\sfw_2)\tilde{\psi}_{21}(\varphi(x,\sfw_2))\psi_{12}(\varphi(x,\sfw_2)),
\end{align*}
for all $x\in\Omega$. Since $\varphi(.,\sfw_2)$ is a non-zero holomorphic function, 
\[
\tilde{\psi}_{21}(\varphi(x,\sfw_2))\psi_{12}(\varphi(x,\sfw_2))=1
\]
for all $x\in \Omega$. Since both $\tilde{\psi}_{21}$ and $\psi_{12}$ are contractive analytic functions on $\D$, the above identity is possible only if $\tilde{\psi}_{21}(z)=\lambda$ and $\psi_{12}(z)=\overline{\lambda}$ for all $z\in\D$, where $\lambda\in\mathbb{T}$. Hence from \eqref{Iden4}, we have $\varphi(x,y)=\overline{\lambda}u(x)$ for all $(x,y)\in \Omega$. Similarly, if we assume that $\varphi(.,\sfw_2)$ is identically zero on $\Omega$ then it would force $\varphi$ to be of the form $\varphi(x,y)=\beta v(y)$ for all $(x,y)\in \Omega$, where $\beta\in \mathbb{T}$. This completes the proof.   
\end{proof}

We now apply the above theorem to the Hardy space over the bidisc and prove Theorem~\ref{Hardy space over the bidisc}.

\textit{Proof of Theorem~\ref{Hardy space over the bidisc}:} Let $\varphi: \D^2\to \D$ be a holomorphic function such that $\varphi(0,0)=\mu\in\D$. Then $b_{\mu}\circ \varphi$ vanishes at the origin, where $b_{\mu}(z)=\frac{z-\mu}{1-\bar{\mu}z}$ is the automorphism of $\D$ corresponding to $\mu\in \D$. 
In view of Proposition~\ref{Iso_CNPs}, the de Branges-Rovnyak space $\clh(\cls_2^{\varphi})$ is a CNP space if and only if  $\clh(\cls_2^{b_{\mu}\circ\varphi})$ is a CNP space. By Theorem~\ref{NCNPm}, $\clh(\cls_2^{b_{\mu}\circ\varphi})$ is a CNP space if and only if either $b_{\mu}\circ \varphi(z_1,z_2)=\lambda z_1 $ or $b_{\mu}\circ \varphi(z_1,z_2)=\lambda z_2 $ for some $\lambda\in \mathbb{T}$. The proof now follows. 
\qed 

The case of tensor products of three or more CNP kernels is significantly different, as we show that there is no non-trivial de Branges-Rovnyak space which is a CNP space. We consider tensor product of three CNP kernels first and then the general case is treated by an inductive argument.  
\begin{Theorem}\label{three kernels}
Let $\ft, \fu, \fv: \Omega\to \D$ be non-constant holomorphic functions such that $\ft(\sfw_1)=\fu(\sfw_2)=\fv(\sfw_3)=0$ for some $\mathbf w=(\sfw_1,\sfw_2,\sfw_3)\in \Omega^3$. Consider the kernel  $K=K_{\ft}\otimes K_{\fu}\otimes K_\fv$ on $\Omega \times \Omega\times \Omega$. There is no non-trivial de Branges-Rovnyak subspace of $\clh(K)$ which is a CNP space. 
\end{Theorem}

\begin{proof}
For the sake of contradiction, assume that $\varphi:\Omega^3 \to \D$ is a non-constant holomorphic function in $\clm_1(\clh(K))$ such that $\clh(K^{\varphi})$ is a CNP space. By Proposition~\ref{Iso_CNPs}, we also assume that $\varphi(\mathbf w)=0$. For $\x=(x_1, x_2, x_3),\mathbf y=(y_1, y_2, y_3)\in\Omega^3$, note that 
\[
1-\frac{1}{K(\x,\mathbf y)}= g(\x)g(\mathbf y)^*-f(\x)f(\mathbf y)^*,
\]
where $g:\Omega^3\to \clb(\C^4,\C)$ and $f:\Omega^3\to \clb(\C^3,\C)$ are given by 
\[
g(x_1,x_2,x_3)=(\ft(x_1),\fu(x_2),\fv(x_3), \ft(x_1)\fu(x_2)\fv(x_3))
\]
and 
\[f(x_1,x_2,x_3)=(\ft(x_1)\fu(x_2), \fu(x_2)\fv(x_3), \ft(x_1)\fv(x_3))
\]
for all $x_1,x_2,x_3\in\Omega$. Then by Theorem~\ref{NCNP}, there exists $\Psi\in H^{\infty}_1(\D, \clb(\C^4))$ such that
\begin{align}\label{CNP_Id_Poly}
f(\x)=g(\x) \psi_1(\varphi(\x))\quad \text{and} \quad \varphi(\x)=g(\x) \psi_2(\varphi(\x)) \quad (\x\in \Omega^3),
 \end{align} 
where for all $z\in\D$, $\Psi(z)=\begin{bmatrix} \psi_1(z)& \psi_2(z)\end{bmatrix}$, $\psi_1(z)\in\clb(\C^3,\C^4)$ and $\psi_2(z)\in \clb(\C,\C^4)$. We denote the slice functions of $\varphi$ by $\varphi_1,\varphi_2,\varphi_3: \Omega^2\to \D$ which are defined as 
\[
\varphi_1(x_2,x_3)=\varphi(\sfw_1, x_2,x_3), \, \varphi_2(x_1,x_3)=\varphi(x_1,\sfw_2,x_3), \text{ and }
\varphi_3(x_1,x_2)=\varphi(x_1,x_2,\sfw_3)
\]
for $x_1,x_2, x_3\in\Omega$. 
Then the crucial observation is that if any of the slice functions is non-zero, say $\varphi_1$, then the de Branges-Rovnyak subspace $\clh((K_\fu\otimes K_{\fv})^{\varphi_1})$ of $\clh(K_\fu\otimes K_{\fv})$ is a CNP space. This follows from ~\eqref{CNP_Id_Poly} and the observation that $\varphi_1$ is an element of $\clm_1(K_\fu\otimes K_{\fv})$. Thus if $\varphi_1$ is non-zero, by Theorem~\ref{NCNPm}, 
\[
\varphi_1(x_2,x_3)=\varphi(\sfw_1,x_2,x_3)=\lambda \fu(x_2) \text{ or } \lambda \fv(x_3)\quad (x_2,x_3\in\Omega).
\]
for some $\lambda\in \mathbb{T}$. We infer similar conclusions if $\varphi_2$ and $\varphi_3$ are non-zero. Next we show that at least one of the slice functions is non-zero and at least one of the slice functions is identically zero. This is a consequence of two claims below. 

\textbf{Claim I:} Not all of the slice functions can be non-zero.

\textit{Proof of claim I.} Suppose on the contrary that all the slice functions are non-zero. Then by the discussion prior to the claim, we know that $\varphi_i$ depends only on one variable for all $i=1,2,3$. This can not happen simultaneously. We show this by making a particular choice, and leave the rest to the reader as the proof is similar.   Let us assume that 
\[
 \varphi_1(x_2,x_3)=\lambda \fu(x_2), \ \varphi_2(x_1,x_3)=\beta \fv (x_3)\ \text{ and } \varphi_3(x_1,x_2)= \gamma \ft(x_1)\quad (x_1,x_2,x_3\in\Omega).\]
Then $\varphi_2(\sfw_1,x_3)=\varphi(\sfw_1,\sfw_2,x_3)=\varphi_1(\sfw_2,x_3)=0$ for all $x_3\in\Omega$. Therefore by the choice of slice function $\varphi_2$, we have $\fv$ vanishes identically, which is a contradiction.  

\textbf{Claim II:} All the slice functions are not identically zero.

 \textit{Proof of claim II.} Suppose all the $\varphi_i$'s are identically zero. Expanding the first identity in ~\eqref{CNP_Id_Poly}, we get  
\begin{align}\label{CNP_Id_Poly3}
\ft(x_1)\fu(x_2)=\ft(x_1) \psi_{11}(\varphi(\x)) + \fu(x_2) \psi_{21}(\varphi(\x)) + \fv(x_3)\psi_{31}(\varphi(\x)) + \ft(x_1)\fu(x_2)\fv(x_3)\psi_{41}(\varphi(\x))
\end{align}
for all $\x=(x_1,x_2,x_3)\in\Omega^3$, where the first column for $\psi_1(z)$ is
\[\begin{bmatrix} \psi_{11}(z)& \psi_{21}(z) & \psi_{31}(z) & \psi_{41}(z)\end{bmatrix}^T\ \text{ and } \psi_{i1}(z)\in\clb(\C)\] 
for all $z\in\D$ and $i=1,\dots,4$. Now taking $x_1=\sfw_1$, $x_3=\sfw_3$ in \eqref{CNP_Id_Poly3} and using the fact that $\varphi(\sfw_1, x_2,\sfw_3)=0$ for any $x_2\in\Omega$, we have $\psi_{21}(0)=0$. Similarly, by choosing $x_2=\sfw_2$ and $x_3=\sfw_3$, we have $\psi_{11}(0)=0$. Therefore, by Schwarz's lemma, $\psi_{i1}(z)=z\tilde{\psi}_{i1}(z)$ for some $\tilde{\psi}_{i1}\in H^\infty_1(\D)~ (i=1,2).$ With all these identities, ~\eqref{CNP_Id_Poly3} now reduces to 
\begin{align}\label{CNP_Id_Poly4}
\ft(x_1)\fu(x_2)=\ft(x_1) \varphi(\x)\tilde{\psi}_{11}(\varphi(\x)) + \fu(x_2)\varphi(\x) \tilde{\psi}_{21}(\varphi(\x)) +& \fv(x_3)\psi_{31}(\varphi(\x))\nonumber\\ &+ \ft(x_1)\fu(x_2)\fv(x_3)\psi_{41}(\varphi(\x)).
\end{align}
Finally, putting $x_3=\sfw_3$ in \eqref{CNP_Id_Poly4} we have $\ft(x_1)\fu(x_2)=0$ for all $x_1,x_2\in\Omega$, which is a contradiction. This proves our claim.

Thus, without any loss of generality, we assume that $\varphi_1$ is non-zero and $\varphi_2 \equiv 0$. In the rest of the proof we show that this leads to a contradiction. Since $\varphi_1$ is non zero, by the observation made earlier in the proof, 
\[
\varphi_1(x_2,x_3)=\varphi(\sfw_1, x_2, x_3)= \lambda \fu(x_2) \text{ or } \lambda \fv(x_3)\quad (x_2,x_3\in\Omega).
\]
On the other hand, $\varphi_2$ is identically zero forces that $\varphi(\sfw_1, x_2,x_3)= \lambda \fu(x_2)$ for all $x_2,x_3\in\Omega$. Expanding the second identity in ~\eqref{CNP_Id_Poly}, we get 
\begin{align}\label{CNP_Id_Poly2}
\varphi(\x)=\ft(x_1) \psi_{12}(\varphi(\x)) + \fu(x_2) \psi_{22}(\varphi(\x)) + \fv(x_3)\psi_{32}(\varphi(\x)) + \ft(x_1)\fu(x_2)\fv(x_3)\psi_{42}(\varphi(\x))
\end{align}
for all $\x=(x_1,x_2,x_3)\in\Omega^3$, where 
\[\psi_2(z)=\begin{bmatrix} \psi_{12}(z)& \psi_{22}(z) & \psi_{32}(z) & \psi_{42}(z)\end{bmatrix}^T\ \text{ and } \psi_{i2}(z)\in\clb(\C)\] 
for all $z\in\D$ and $i=1,\dots,4$.
Now taking $x_1=\sfw_1$ and $x_3=\sfw_3$ in ~\eqref{CNP_Id_Poly2}, we get
\[
\lambda \fu(x_2)=\fu(x_2) \psi_{22}(\varphi(\sfw_1,x_2,\sfw_3))= \fu(x_2) \psi_{22}(\lambda\fu(x_2))\quad (x_2\in\Omega).
\]
Using holomorphic property of all the functions involved, we conclude that $\psi_{22}\equiv \lambda$. This shows that $\psi_2(z)=\begin{bmatrix} 0& \lambda & 0 & 0\end{bmatrix}^T$ for all $z\in\D$, and therefore by \eqref{CNP_Id_Poly2}, $\varphi(\x)= \lambda \fu(x_2)$ for all $\x\in\Omega^3$. In such a case, the kernel of the de Branges-Rovnyak space $\clh(K^{\varphi})$ is 
\[
K^{\varphi}(\x,\mathbf y)= K_{\ft}(x_1,y_1)K_{\fv}(x_3,y_3)=\frac{1}{(1-\ft(x_1)\overline{\ft(y_1)})(1-\fv(x_3)\overline{\fv(y_3)})} \quad (\x,\mathbf y\in\Omega^3).
\]
which is not a CNP kernel. Indeed, by the open mapping theorem, we can find $x_1, x_2, y_1,y_2\in\Omega$ such that $\ft(x_1)=-\ft(y_1)=t$ and $\fv(x_2)=-\fv(y_2)=v$ for some non-zero real numbers $t$ and $v$. Now for the choice of points $w_1=(x_1,x_2)$ and $w_2=(y_1,y_2)$,
\[
\begin{bmatrix}
(1-\frac{1}{K^{\varphi}})(w_i,w_j)\end{bmatrix}^2_{i,j=1}=\begin{bmatrix} 1-(1-t^2)(1-v^2)& 1-(1+t^2)(1+v^2)\\
1-(1+t^2)(1+v^2)&1-(1-t^2)(1-v^2)\end{bmatrix} 
\]
 is not positive semi-definite.
This contradicts our assumption and also completes the proof.  
\end{proof}

By now, the reader must be convinced that the above result is also true for tensor products of more than three kernels, as one can reduce the number of kernels by considering appropriate slice functions. 

\begin{Theorem}
Let $n>3$, and let $\fu_i:\Omega^n\to \D$ be non-constant holomorphic functions such that $\fu_i(\sfw_i)=0$ $(1\le i\le n)$ for some $\mathbf w=(\sfw_1,\ldots, \sfw_n)\in \Omega^n$.  Consider the kernel  $K=\otimes_{i=1}^n K_{\fu_i}$ on $\Omega^n$. There is no non-trivial de Branges-Rovnyak subspace of $\clh(K)$ which is a CNP space. 
\end{Theorem}
\begin{proof}
Suppose that $\varphi:\Omega^n\to \D$ is a non-constant holomorphic multiplier in $\clm_1(\clh(K))$ such that $\varphi(\mathbf w)=0$ and $\clh(K^{\varphi})$ is a CNP space. Define the slice functions of $\varphi$, for each $i=1,\ldots, n$, by 
\[\varphi_i :\Omega^{n-1}\to \D,\ \varphi_i(x_1, \ldots, x_{n-1})=\varphi( x_1,\ldots, x_{i-1}, \underset{i\text{th place}}{\sfw_i}, x_{i+1},\ldots, x_{n-1}).\]
Then, by the same way as it is done in the proof of Theorem~\ref{three kernels}, one shows that all the $\varphi_i$'s are not identically zero. Moreover, if $\varphi_i$ is non-zero then the de Branges-Rovnyak subspace corresponding to $\varphi_i$ in $\clh(\otimes_{j\neq i} K_{\fu_j})$ is a CNP space.
Thus by a simple induction argument and Theorem~\ref{three kernels}, such a $\varphi$ does not exist. This completes the proof.  
\end{proof}

 As an immediate consequence of the above two results, we have the following corollary for the Hardy space over the polydisc. 
\begin{Corollary}
For $n\geq 3$, there is no non-constant $\varphi\in H^\infty_1(\D^n)$ such that 
the de Branges-Rovnyak space $\clh(\cls_n^\varphi)$ corresponding to the Szeg\"o kernel $\cls_n$ on $\D^n$ is a CNP space.
\end{Corollary}

\subsection*{Schur product of CNP kernels}
 Let $K_1$ and $K_2$ be kernels on $\Omega$. Then the Schur product of $K_1$ and $K_2$ is a kernel on $\Omega$ denoted by $K_1\circ K_2$ and defined as 
   \[
   (K_1\circ K_2)(x,y)= K_1(x,y)K_2(x,y)\quad (x, y\in\Omega).
   \]
The positive semi-definite property of $K_1\circ K_2$ follows from Schur's theorem, which says that Schur product (entry-wise product) of positive semi-definite matrices is positive semi-definite (see ~\cite{PaulRaghu}). We consider the case of Schur product of CNP kernels in the following theorem.

\begin{Theorem}\label{NCNPs}
Let $u:\Omega\to \D$ be a non-constant holomorphic function such that $u(\mathbf w)=0$ for some $\mathbf w\in \Omega$. Consider the kernel $K=K_u\circ K_u$. Assume that $\varphi:\Omega\to \D$ is a non-constant holomorphic function in $\clm_1(\clh(K))$ with $\varphi(\mathbf w)=0$. Then the de Branges-Rovnyak space $\clh(K^{\varphi})$ is a CNP space if and only if $\varphi(x)=\lambda u(x)$ for some $\lambda\in \mathbb{T}$.  

\end{Theorem}
\begin{proof}
For all $x,y\in\Omega$, note that 
\[
1-\frac{1}{K(x,y)}=g(x)g(y)^*-f(x)f(y)^*,
\]
where $f,g: \Omega\to \clb(\C,\C)$ are given by 
\[
g(x)=\sqrt{2} u(x)\ \text{ and } f(x)= u(x)^2 \quad (x\in\Omega).
\]
Then by Theorem~\ref{NCNP}, $\clh(K^{\varphi})$ is a CNP space if and only if there exists $\Psi\in H^{\infty}_1(\D, \clb(\C^2,\C))$ such that 
\begin{equation}\label{Identity1}
u(x)=\sqrt{2}\psi_1(\varphi(x))\ \text{ and } \varphi(x)=\sqrt{2}u(x)\psi_2(\varphi(x)),\quad (x\in\Omega)
\end{equation}
where $\Psi(z)=\begin{bmatrix}\psi_1(z)& \psi_2(z) \end{bmatrix}$ for all $z\in\D$.
By ~\eqref{Identity1}, $2\psi_1(\varphi(x))\psi_2(\varphi(x))=\varphi(x)$ for all $x\in\Omega$. The holomorphic function $\varphi$, being non-constant, is an open map. Then the above identity holds for all $z\in \D$, that is 
\begin{equation}\label{Identity 2}
2\psi_1(z)\psi_2(z)=z \quad (z\in\D).
\end{equation}

\textbf{Claim:} The functions $\sqrt{2}\psi_1$ and $\sqrt{2}\psi_2$ are inner functions. 

\textit{Proof of the claim.} Since $|\psi_1(z)|^2+|\psi_2(z)|^2\leq 1$ for all $z\in\mathbb D$, then using ~\eqref{Identity 2} we have 
\begin{align*}
|\psi_1(z)|^2+ \frac{|z|^2}{4|\psi_1(z)|^2} \leq 1, \quad (z\in \D, z\neq 0)
    \end{align*}
that is, 
\begin{align}\label{Identity 3}
1-\sqrt{1-|z|^2}\leq 2|\psi_1(z)|^2\leq1+\sqrt{1-|z|^2}\quad (z\in \D, z\neq 0).
\end{align}
Since $\psi_1 \in H^\infty(\D)$, the non-tangential $\lim_{z\to e^{i\theta}}|\psi_1(z)|$ exists for almost all $e^{i \theta}\in\mathbb T$. Hence, by \ref{Identity 3}, $\lim_{z\to e^{i\theta}}\sqrt{2}|\psi_1(z)|=1$ for almost all $e^{i \theta}\in\mathbb T$. In other words, $\sqrt{2}\psi_1$ is an inner function. Therefore it follows from ~\eqref{Identity 2} that $\sqrt{2}\psi_2$ is also an inner function, and hence the claim is proved.  

The first identity of ~\eqref{Identity1} shows that $\psi_1(0)=0$. The identity ~\eqref{Identity 2} and $\psi_1(0)=0$, that is, $z$ is a product of two inner functions, holds only if 
\[
\sqrt{2}\psi_2(z)=\lambda\ \text{ and } \sqrt{2}\psi_1(z)=\bar{\lambda} z\quad (z\in\D)\]
for some $\lambda\in \mathbb{T}$. The proof now follows from the second identity of ~\eqref{Identity1}.
\end{proof}
 By applying the above result to the Bergman space over the disc, which has the kernel
 \[B(z,w)=\frac{1}{(1-z \overline{w})^2},\quad (z,w\in \D)
 \]
 we get the following result. 
\begin{Corollary}\label{BergCNP}
Let $\varphi$ be a non-constant function in $H^\infty_1(\D)$ such that $\varphi\in\clm_1(\clh(B))$.
 Then the de Branges-Rovnyak subspace $\clh(B^{\varphi})$ of the Bergman space over $\D$ is a CNP space if and only if $\varphi(z)=\lambda b_{\mu}(z)$ for some $\lambda\in \mathbb{T}$ and an automorphism $b_{\mu}$ of the unit disc. 
\end{Corollary}
Recall that for each real number $\alpha>-1$, the corresponding weighted Bergman spaces on the unit disc has the kernel
\[
B_{\alpha}(z,w)=\frac{1}{(1-z\overline{w})^{2+\alpha}}\quad (z,w\in\D).
\]

\begin{Theorem}\label{BergNonCNP}
For any real number $\alpha\geq1$, there is no non-constant $\varphi\in H^\infty_1(\D)$ such that the de Branges-Rovnyak space $\clh(B^\varphi_\alpha)$ of the weighted Bergman space $\clh(B_{\alpha})$ is a CNP space.
\end{Theorem}
\begin{proof}
Without loss of generality we assume that $\varphi(0)=0$. Let $p=2+\alpha$. For all $z,w\in\D$, note that 
    \[1-\frac{1}{B_\alpha(z,w)}=g(z)g(w)^*-f(z)f(w)^*\]
where $f$ and $g$ are given by 
\[g(z)=\left(\sqrt{p}z,\sqrt{\frac{p(p-1)(p-2)}{6}}z^3,\ldots \right),\]
and 
\[f(z)=\left(\sqrt{\frac{p(p-1)}{2}}z^2, \sqrt{\frac{p(p-1)(p-2)(p-3)}{4!}}z^4\ldots \right).\]
The remaining terms in the expression of $f$ and $g$ depends on the exact value of $p$. For our purpose the first two terms of both $f$ and $g$ will be enough. Let us write that 
 $f(z)\in\clb(\cle,\C)$ and $g(z)\in\clb(\clf,\C)$, for all $z\in\D$, where $\cle$ and $\clf$ are Hilbert spaces and depends entirely on $p$. 
    Then by Theorem~\ref{NCNP}, $\clh(B^{\varphi}_\alpha)$ is a CNP space if and only if there exists $\Psi\in H^{\infty}_1(\D, \clb(\cle\oplus\C,\clf))$ such that 
    \begin{align}\label{BergNonCNPId1}
      f(z)=g(z)\psi_1(\varphi(z))\quad \text{and}\quad \varphi(z)=g(z)\psi_2(\varphi(z))\quad (z\in\D)  
    \end{align}
    where for all $z\in\D$, $\Psi(z)=[\psi_1(z)\ \psi_2(z)]$, $\psi_1(z)\in\clb(\cle,\clf)$, and $\psi_2(z)\in\clb(\C,\clf)$.
    Let $\psi_1(z)=[\psi^{(1)}_{ij}]$.
    Now from the first identity in ~\eqref{BergNonCNPId1}, we get
    \begin{align}\label{BergNonCNPId2}
        \sqrt{\frac{p(p-1)}{2}}z=\sqrt{p}\psi^{(1)}_{11}(\varphi(z))+\sqrt{\frac{p(p-1)(p-2)}{6}}z^2\psi^{(1)}_{21}(\varphi(z))+\cdots.
    \end{align}
    Taking $z=0$ in ~\eqref{BergNonCNPId2}, we get $\psi^{(1)}_{11}(0)=0$, and thus, for some $\tilde{\psi}_{11}\in H^\infty_1(\D)$, $\psi^{(1)}_{11}(z)=z\tilde{\psi}_{11}(z)$. Again, since $\varphi(0)=0$, $\varphi(z)=z\varphi_1(z)$ for some $\varphi_1\in H^\infty_1(\D)$. Now from ~\eqref{BergNonCNPId2}, we have
    \begin{align}\label{gene_id}
        \sqrt{\frac{p(p-1)}{2}}=\sqrt{p}\varphi_1(z)\tilde{\psi}_{11}(\varphi(z))+\sqrt{\frac{p(p-1)(p-2)}{6}}z\psi^{(1)}_{21}(\varphi(z))+\cdots.
    \end{align}
    Again taking $z=0$ in \eqref{gene_id}, we have 
    $$\sqrt{\frac{p-1}{2}}=\varphi_1(0)\tilde{\psi}_{11}(0).$$
    So, $$|\varphi_1(0)\tilde{\psi}_{11}(0)|=\sqrt{\frac{p-1}{2}}\geq1,$$
    which is a contradiction for all $p>3$. For $p=3$, $|\varphi_1(0)\tilde{\psi}_{11}(0)|=1$ implies that $\varphi_1(z)=\gamma$ $(z\in\D)$ for some $\gamma\in \mathbb{T}$. Thus $\varphi(z)=\gamma z$ for all $z\in\D$. This is a contradiction as $B_{1}^{\varphi}(z,w)=\frac{1}{(1-z\bar{w})^2}$ is not a CNP kernel. This completes the proof. 
\end{proof}

It has been shown in ~\cite[Theorem A]{Zhu} that if $\varphi\in H^{\infty}_1(\D)$ is a finite Blaschke product then the de Branges-Rovnyak space $\clh(B^{\varphi})$, which is also known as sub-Bergman space associated to $\varphi$, is same as $H^2(\D)$ but with a possibly different norm. However, from Corollary ~\ref{BergCNP}, we conclude that for a finite Blaschke product $\varphi$, $\clh(B^{\varphi})$ is isometrically isomorphic to $H^2(\D)$ as reproducing kernel Hilbert spaces if and only if $\varphi$ is a Blaschke factor. For a general $\varphi\in H^{\infty}_1(\D)$ or even when $\varphi$ is a general inner function, explicit description of the de Branges-Rovnyak space $\clh(B^{\varphi})$ is not known (see Section $6$ of ~\cite{Zhu} for a detailed discussion on this issue). However, by Corollary ~\ref{BergCNP}, we conclude that they are never CNP spaces.

\vspace{0.1in} \noindent\textbf{Conflict of interest:}
The author states that there is no conflict of interest. No data sets were generated or analyzed during the current study.

\vspace{0.1in} \noindent\textbf{Acknowledgement:}
The authors thank the referees for their valuable comments and suggestions, which helps to improve the article. The first author is supported by Prime Minister's Research Fellowship, ID: 1303115.
The second author is supported by the Mathematical Research Impact Centric Support (MATRICS) grant, File No: MTR/2021/000560, by the Science and Engineering Research Board (SERB), Department of Science \& Technology (DST), Government of India.

\bibliographystyle{amsplain}

\begin{thebibliography}{99}
\bibitem{AM}
M. B. Abrahamse, {\em The Pick interpolation theorem for finitely connected domains}, Michigan Math. J. 26 (1979), 195 - 203.


\bibitem{AglerM}
J. Agler and J. E. McCarthy, {\em Pick Interpolation and Hilbert Function Spaces}, American Mathematical
Society, Providence, 2002.

\bibitem{AglerY}
J. Agler and N. J. Young, {\em Realization of functions on the symmetrized bidisc}, J. Math. Anal. Appl. 453 (2017), 227 - 240. 

\bibitem{AHMR}
A. Aleman, M. Hartz, J. E. McCarthy, and S. Richter, {\em Free outer functions in complete Pick spaces}, Trans. Amer. Math. Soc. 376 (2023), 1929 - 1978.


\bibitem{ABJK}
D. Alpay, T. Bhattacharyya, A. Jindal, and P. Kumar,
{\em Complete Nevanlinna-Pick kernels, the Schwarz lemma and the Schur algorithm}, arXiv:2312.01287.

\bibitem{BallT}
J. A. Ball and T. T. Trent, {\em Unitary colligations, reproducing kernel Hilbert spaces, and Nevanlinna-
Pick interpolation in several variables},
J. Funct. Anal. 157 (1998), 1 - 61. 


\bibitem{Ball}
J. A. Ball, T. Trent, and V. Vinnikov, {\em Interpolation and Commutant Lifting for Multipliers on
Reproducing Kernel Hilbert Spaces},  Operator Theory: Advances and
Applications 122 (2001), 89 – 138. Springer Basel AG.

\bibitem{BhattSau}
T. Bhattacharyya and H. Sau, {\em Holomorphic functions on the symmetrized bidisk - realization, interpolation and extension}, J. Funct. Anal. 274 (2018), 504 – 524.



\bibitem{CGR}
N. Chevrot, D. Guillot, and T. Ransford, {\em De Branges-Rovnyak spaces and Dirichlet spaces}, J. Funct. Anal. 259 (2010), 2366 – 2383.

\bibitem{Chu}
C. Chu, {\em Which de Branges-Rovnyak spaces have complete nevanlinna-pick
property?} J. Funct. Anal. 279 (2020), 1 - 15. 


\bibitem{BranRov66}
L. de Branges and J. Rovnyak, {\em Square Summable Power Series}, Holt, Rinehart and Winston, New York, 1966.

\bibitem{BranRov68}
L. de Branges and J. Rovnyak, {\em Appendix on square summable power series, Canonical models in quantum scattering theory, Perturbation Theory in Quantum Mechanics (C. M. Wilcox, ed.)}, Wiley, New York, 1968.

\bibitem{FKMR}
O. El-Fallah, K. Kellay, H. Klaja, J. Mashreghi, and T. Ransford, 
{\em Dirichlet spaces with superharmonic weights and de Branges–Rovnyak spaces},
Complex Anal. Oper. Theory 10 (2016), no.1, 97 – 107.


\bibitem{Hartz}
M. Hartz, {\em Every complete Pick space satisfies the column-row property}, Acta Math. 231 (2023), 345 - 386.

\bibitem{Hartz_1}
M. Hartz, {\em On the isomorphism problem for multiplier algebras of Nevanlinna-Pick spaces}, Canad. J. Math. 69 (2017), 54 - 106.

\bibitem{JurryK}
M. T. Jury, G. Knese, and S. McCullough, {\em Nevanlinna-Pick interpolation on distinguished varieties
in the bidisk}, J. Funct. Anal. 262 (2012), 3812 - 3838.

 
\bibitem{JuryMartin}
M. T. Jury and R. T. W. Martin, {\em The Smirnov classes for the Fock space
and complete Pick spaces}, Indiana Univ. Math. J. 70 (2021), 269 - 284,


 \bibitem{Leech}
R. Leech, {\em Factorization of analytic functions and operator inequalities}, Integr. Equ. Oper. Theory 78 (2014), 71 - 73.

\bibitem{LuoGR}
S. Luo, C. Gu, and S. Richter, {\em Higher order local Dirichlet integrals and de Branges-Rovnyak spaces}, Adv. Math. 385 (2021), Paper No. 107748, 47 pp.

\bibitem{LuoZhu}
S. Luo and  K. Zhu, {\em Sub-Bergman Hilbert spaces on the unit disk III}, Canad. J. Math. (to appear),  arXiv:2302.01980.


 \bibitem{Mccull92}
 S. McCullough, {\em Caratheodory interpolation kernels}, Integr. Equ. Oper. Theory 15 (1992), 43 - 71.

\bibitem{McTr00}
S. McCullough and T. T. Trent, {\em Invariant subspaces and Nevanlinna–Pick kernels}, J. Funct. Anal. 178 (2000), 226 - 249.


 \bibitem{Nevanlinna1} 
 R. Nevanlinna, \emph{Über beschränkte Funktionen, die in gegebenen Punkten vorgeschrieben Werte annehmen}, Ann. Acad. Sci. Fenn. Ser. A 13 (1919), no. 1.



\bibitem{PaulRaghu}
	V. I. Paulsen and M. Raghupathi, {\em An introduction to the theory of reproducing kernel Hilbert spaces}, Cambridge University Press, Cambridge, 2016.


 \bibitem{Pick} 
G. Pick, \emph{Über die Beschränkungen analytischer Funktionen, welche durch vorgegebene Funktionswerte bewirkt werden}, Math. Ann. 77 (1916), 7 - 23.

\bibitem{Quiggin}
P. Quiggin, {\em For which reproducing kernel Hilbert spaces is Pick’s theorem true?} Integral Equ. Oper.
Theory 16 (2) (1993), 244 - 266.

\bibitem{Sara}
	D. Sarason, {\em Generalized interpolation in $H^{\infty}$}, Trans. Amer. Math. Soc. 127 (1967), 179 - 203.

\bibitem{Sarason}
D. Sarason, {\em Sub-Hardy Hilbert Spaces in the Unit Disk}, University of Arkansas Lecture Notes, Wiley, New York, 1994.


\bibitem{Jesse}
J. G. Sautel, \href{https://trace.tennessee.edu/utk_graddiss/7144/}{\em Some Results About Reproducing Kernel Hilbert Spaces of Certain Structure},  PhD 
diss., University of Tennessee, 2022. 

\bibitem{Shimo}
S. Shimorin, {\em Complete Nevanlinna-Pick property of Dirichlet-type spaces},
J. Funct. Anal. 191 (2002), 276 - 296.

\bibitem{Zhu}
K. Zhu, {\em Sub-Bergman Hilbert spaces in the unit disk. II},
J. Funct. Anal. 202 (2003), 327 – 341.

 
\end{thebibliography}

\end{document}